\documentclass[leqno,12pt]{amsart} 
\usepackage[top=30truemm,bottom=30truemm,left=25truemm,right=25truemm]{geometry}
\usepackage{amssymb}
\usepackage{amsmath}
\usepackage{amsthm}
\usepackage{amscd}
\usepackage{mathrsfs}
\usepackage{graphicx}
\usepackage[dvips]{color}
\usepackage[all]{xy}
\usepackage{url}
\usepackage{comment}
\theoremstyle{plain} 
\newtheorem{Thm}{Theorem}[section]
\newtheorem{Lem}[Thm]{Lemma}
\newtheorem{Prop}[Thm]{Proposition}

\theoremstyle{definition} 
\newtheorem{definition*}{定義}
\newtheorem{Def}[Thm]{Definition}
\newtheorem{Rem}[Thm]{Remark}
\newtheorem{Ex}[Thm]{Example}
\begin{document}
	
	\title[Higher dimensional case of sharper estimate]{Sharper estimates of Ohsawa--Takegoshi $L^2$-extension theorem in higher dimensional case}
	
	\author[S. Kikuchi]{Shota Kikuchi} 
	
	%
	\address{
		Graduate School of Mathematics, Nagoya University, Furocho, Chikusa-ku, Nagoya, 464-8602, Japan
	}
	\email{m16015w@math.nagoya-u.ac.jp}
	\maketitle
	\begin{abstract}
	In \cite{H1}, Hosono obtained sharper estimates of the Ohsawa--Takegoshi $L^2$-extention theorem by allowing the constant depending on the weight function for a domain in $\mathbb{C}$. 
In this article, we show the higher dimensional case of sharper estimates of the Ohsawa--Takegoshi $L^2$-extention theorem. 
	To prove the higher dimensional case of them, we establish an analogue of Berndtsson--Lempert type $L^2$-extension theorem by using the pluricomplex Green functions with poles along subvarieties.  
As a special case, we consider the sharper estimates in terms of the Azukawa pseudometric and show that the higher dimensional case of sharper estimate provides the $L^2$-minimum extension for radial case. 
	\end{abstract}
	
\section{Introduction}
	
	The {\it{Ohsawa--Takegoshi $L^2$-extension theorem}} \cite{OT} states the following: Let $\Omega \subset \mathbb{C}^n$ be a bounded pseudoconvex domain, $V$ a closed complex submanifold of $\Omega$ and $\varphi$ a plurisubharmonic function on $\Omega$. Suppose that a holomorphic function $f \in \mathcal{O}(V)$ is given. 
	Then there exists a holomorphic function $F \in \mathcal{O}(\Omega)$ satisfying the $L^2$-estimate $\int_{\Omega}|F|^2 e^{-\varphi} \le C \int_V |f|^2 e^{-\varphi}$, where $C$ is the positive constant that is independent of the weight $\varphi$ and a given $f \in \mathcal{O}(V)$. 

	In \cite{BL}, \cite{Blo} and \cite{GZ}, the {\it{optimal $L^2$-extension theorem}} was proved. This means that we can determine the positive constant $C$ in the  best possible way. 
	In particular, Berndtsson--Lempert \cite{BL} proved the optimal $L^2$-estimate under the assumption that there exists a negative plurisubharmonic function $G$ on $\Omega$ such that 
	\begin{equation}\label{Greentype-condition}
	\log d_V^2(z) - B(z) \le G(z) \le \log d_V^2(z) + A(z), 
	\end{equation}
where $d_V(z)$ is the distance between $z \in \Omega$ and $V$, $A(z)$ and $B(z)$ are continuous functions on $\Omega$. 
	The function $B(z)$ appears in the $L^2$-estimate as follows:   
	A negative plurisubharmonic function $G$ satisfying $(\ref{Greentype-condition})$ is called a Green-type function on $\Omega$ with poles along $V$ \cite{H1}.  
	Characterizing pairs $(\Omega,V)$ of a domain $\Omega \subset \mathbb{C}^n$ and a closed subvariety $V \subset \Omega$ admitting a Green-type function with poles along $V$ in terms of their geometry is an open problem.  

	Hosono proposed in \cite{H1} an idea of getting an $L^2$-estimate sharper than the one of Berndtsson--Lempert type $L^2$-extension theorem by allowing constants depending on weight functions. 
	In other words, we can determine the positive constant $C'$ depending on weight functions sharper than the optimal constant $C$. 
	In this article, we call this result {\it{Hosono's sharper estimate}}. 
	Specifically, the idea used in the proof of Hosono's sharper estimate is as follows: 
	Let $\Omega \subset \mathbb{C}$ be a bounded domain containing $0$ and $\varphi$ a subharmonic function with $\varphi(0)=0$. 
	At first, we construct the pseudoconvex domain $\tilde{\Omega}=\{(z,w) \in \mathbb{C}^{2} : z \in \Omega, |w|^2 < e^{-\varphi(z)} \}$ in $\mathbb{C}^2$. 
	For the pseudoconvex domain $\tilde{\Omega}$ and a closed complex submanifold $\tilde{V}=\{z=0\}$ of it, we construct the Green-type function $\tilde{G}$ using solutions of the Dirichlet problem for complex Monge--Amp\`{e}re equation. 
	By using Berndtsson--Lempert type $L^2$-extension theorem for $\tilde{\Omega}$, $\tilde{V}$ and $\tilde{G}$, we can get an $L^2$-estimate sharper than the one of Berndtsson--Lempert type $L^2$-extension theorem. 
	As an application of Hosono's sharper estimate, in the case where $\Omega$ is the unit disc $\{|z|<1\}$ in $\mathbb{C}$ and $\varphi$ is a radial subharmonic function, 
	Hosono was able to determine the $L^2$-minimum extension of the function $1$ on the subvariety $\{0\}$.  

	In this article, for a bounded pseudoconvex domain in $\mathbb{C}^n$ and a closed complex submanifold of it with some conditions, we generalize Hosono's sharper estimate.
	In general, for a closed submanifold, it is difficult to construct the Green-type function because we do not know whether the logarithmic distance from a closed complex submanifold is plurisubharmonic. 
	To generalize this result, we establish an analogue of Berndtsson--Lempert type $L^2$-extension theorem by using the theory of the pluricomplex Green function with poles along subvarieties \cite{Ras}. Our first main theorem is as follows. 

\begin{Thm}[Theorem \ref{theoremOT}]\label{Improve}

	Let $\Omega$ be a bounded pseudoconvex domain in $\mathbb{C}^n$, $V$ a closed complex submanifold of $\Omega$ with codimension $k$ such that $V$ has bounded global generators $\psi=(\psi_1,\ldots,\psi_k)$ and there exists a positive constant $C$ such that $\displaystyle \frac{1}{C} \le |J_{\psi}|$ near $V$ where $J_{\psi}$ is a Jacobian of $\psi$ for suitable coordinates. 
	We state about $J_{\psi}$ in Section 2.
	Let $\varphi$ be a plurisubharmonic function on $\Omega$ and $G_{\Omega,V}$ the pluricomplex Green function on $\Omega$ with poles along $V$. 
	Assume that there exist some continuous function $B$ on $\Omega$ such that  
	\begin{equation*}
	\log |\psi(z)| - B(z) \le G_{\Omega,V}(z). 
	\end{equation*}
	Then for any holomorphic function $f$ on $V$ with $\displaystyle \int_V |f|^2 e^{-\varphi + 2kB} < \infty$, there exists a holomorphic function $F$ on $\Omega$ such that $F|_V = f$ and
	\begin{equation*}
	\int_\Omega |F|^2 e^{-\varphi} \le C \sigma_k \int_V |f|^2 e^{-\varphi + 2kB}, 
	\end{equation*}
where $\sigma_k$ is the volume of the unit ball in $\mathbb{C}^k$. 
\end{Thm}

	\noindent
	In Section $2$, we give a geometric interpretation for the assumptions of Theorem \ref{Improve}. 

	Under the conditions in Theorem \ref{Improve}, we consider a pseudoconvex domain $\tilde{\Omega}$ in $\mathbb{C}^{n+k}$ defined by 
	\begin{equation*}
	\tilde{\Omega} = \{(z,w) \in \mathbb{C}^{n+k} : z \in \Omega, |w|^2 < e^{-\frac{\varphi(z)}{k}} \} 
	\end{equation*}
and a closed complex submanifold $\tilde{V}$ of $\tilde{\Omega}$ such that
	\begin{equation*}
	\tilde{V} = \{\tilde{\psi}_1 = \cdots = \tilde{\psi}_k = 0 \}, 
	\end{equation*}
where $\tilde{\psi}_i(z,w) := \psi_i(z) $ are holomorphic functions on $\tilde{\Omega}$. 
	Let $\tilde{G}$ be the pluricomplex Green function on $\tilde{\Omega}$ with poles along $\tilde{V}$ or a subsolution of it such that there exists a continuous function $\tilde{B}(z,w)$ on $\tilde{\Omega}$ such that 
	\begin{equation*}
	\log|\tilde{\psi}(z,w)| - \tilde{B}(z,w) \le \tilde{G}(z,w).
	\end{equation*}
	Then, by using Theorem \ref{Improve}, we can obtain the following higher dimensional case of Hosono's sharper estimate. 

\begin{Thm}[Theorem \ref{theoremH}]\label{Higher}

	Under the above setting, the following statements hold.   

	$(1)$ For any holomorphic function $f$ on $V$ with $\displaystyle \int_V |f|^2 e^{-\varphi+2kB} < \infty$, there exists a holomorphic function $F$ on $\Omega$ such that $F|_V = f$ and
	\begin{equation*}
	\int_{\Omega} |F|^2 e^{-\varphi} \le C \int_{\tilde{V}} |\tilde{f}|^2 e^{2k\tilde{B}}, 
	\end{equation*}
where C is a positive constant determined from Theorem \ref{Improve} and the holomorphic function $\tilde{f}$ on $\tilde{V}$ is defined by $\tilde{f}(z,w) := f(z)$. 

	$(2)$ Suppose that $\tilde{\Omega}$ is a strictly pseudoconvex domain and $-B(z)$ is a plurisubharmonic function. Then one can make the estimate in $(1)$ strictly sharper than the one in Theorem \ref{Improve}, i.e., there exist functions $\tilde{G}$ and $\tilde{B}$ satisfying the above conditions such that the estimate
	\begin{equation*}
	\int_{\tilde{V}} |\tilde{f}|^2 e^{2k\tilde{B}} < \sigma_k \int_V |f|^2 e^{-\varphi + 2kB} 
	\end{equation*}
holds.
\end{Thm}

	\noindent
	As an application of Theorem \ref{Higher}, we can show that for a unit ball $\Omega = \mathbb{B}^n$ in $\mathbb{C}^n$, a closed complex submanifold $V = \{z_1 = \cdots = z_k = 0\} = \{z'=0\}$ and a radial plurisubharmonic function $\varphi$ with respect to $V$, i.e., $\varphi(z) = \varphi(|z'|)$, one can obtain the $L^2$-minimum extension of holomorphic functions $f$ on $V$. 

	In addition, we consider the sharper estimates in terms of the Azukawa pseudometric. 
	Specifically, we aim at the comparison with the result which was obtained in \cite{H2}. 
	To consider it, we prove the following result.  

\begin{Thm}[Theorem \ref{Azukawa2}]\label{Azukawa0}

	Let $\Omega$ be a bounded pseudoconvex domain in $\mathbb{C}^n$, $V$ be a closed complex submanifold defined by $V = \{ z_1 = \cdots = z_k = 0 \}$ and $\varphi$ be a plurisubharmonic function on $\Omega$. 
	Let $G_{\Omega,V}$ be the pluricomplex Green function on $\Omega$ with poles along $V$. 
	We assume that there exists the limit
	\begin{equation*}
	A_{\Omega, V, w}(X) := \lim_{\lambda \to 0} \left( G_{\Omega,V}(\lambda X, w) - \log|\lambda| \right), 
	\end{equation*}
where $(0,\ldots,0,w) \in V$ and $0 \neq X \in \mathbb{C}^k$. 
	We define $I_{\Omega, V, w} := \{X \in \mathbb{C}^n | A_{\Omega, V, w}(X) < 0 \}$. 
	Then for any holomorphic function $f$ on $V$ with $\int_V {\rm{vol}}(I_{\Omega, V, w}) |f^2| e^{-\varphi} < \infty$, there exists a holomorphic function $F$ on $\Omega$ such that $F|_V = f$ and
	\begin{equation*}
	\int_{\Omega} |F|^2 e^{-\varphi} \le \int_V {\rm{vol}}(I_{\Omega, V, w}) |f^2| e^{-\varphi}. 
	\end{equation*}	 
\end{Thm}

	\noindent
	Using the above result, we study a sharper estimate in a specific example. 

	This article is organized as follows. 
	In Section 2, we define the pluricomplex Green function with poles along subvarieties and introduce some properties of it. 
	We also prove Theorem \ref{Improve} following the argument of the proof of Berndtsson--Lempert type $L^2$-extension theorem \cite{H3}. 
	In Section 3, we prove Theorem \ref{Higher} by using Theorem \ref{Improve}. 
	In Section 4, as a special case, we consider the sharper estimates in terms of the Azukawa pseudometric. 
	In particular, we prove Theorem \ref{Azukawa0} and study a specific example.
	In addition, as an application of Theorem \ref{Higher}, in the case where $\Omega = \mathbb{B}^n$ is the unit ball in $\mathbb{C}^n$, $V=\{z_1=\cdots=z_k=0\}$ is a closed submanifold and $\varphi$ is a radial plurisubharmonic function with respect to $V$, we prove that one can obtain the $L^2$-minimum extension of holomorphic functions $f$ on $V$.
	

\section{Establishment an analogue of  Berndtsson--Lempert type $L^2$-extension theorem}

Let $\Omega$ be a domain in $\mathbb{C}^n$ and $V$ a anlytic subvariety of $\Omega$. In other words, for any $z \in V$, there exist a neighborhood $U$ of $z$ and holomorphic functions $\psi_1, \ldots, \psi_k$ on $U$ such that $V \cap U = \{\psi_1 = \dots = \psi_k = 0 \}$. 
Let $\mathcal{O}_{\Omega}$ be the sheaf of germs of locally defined holomorphic functions on $\Omega$ and $\mathcal{I}_V$ be the coherent ideal sheaf of $V$ in $\mathcal{O}_{\Omega}$. 

\begin{Def}[\cite{Ras}]

Let $\Omega$ be a domain in $\mathbb{C}^n$ and $V$ a anlytic subvariety of $\Omega$. 
The class $\mathcal{F}_V$ consists of all negative plurisubharmonic functions $u$ on $\Omega$ such that for any $z \in \Omega$, there exist local generators $\psi_1, \ldots, \psi_k$ of $\mathcal{I}_V$ near $z$ and a constant $C$ depending on $u$ and generators such that $u \le \log|\psi| + C$ near $z$ where we denote $\psi = (\psi_1, \ldots, \psi_k)$. 

\end{Def}

\begin{Def}[\cite{Ras}]

{\it{The pluricomplex Green function $G_{\Omega,V}$ with poles along $V$}} is the upper envelope of all functions in $\mathcal{F}_V$, i.e., for any $z \in \Omega$, 
	\begin{equation*}
	G_{\Omega, V}(z) := \sup\{u(z) | u \in \mathcal{F}_V\}. 
	\end{equation*}
\end{Def}

When $\mathcal{F}_V = \varnothing$, we define $G_{\Omega,V} = -\infty$. 
The pluricomplex Green function $G_{\Omega,V}$ have the following important property. 

\begin{Thm}[\cite{Ras}]

If $V$ is closed, then $G_{\Omega, V} \in \mathcal{F}_V$. 
\end{Thm}

In other words, if $V$ is closed, $G_{\Omega,V}$ is plurisubharmonic without upper semi-continuous regularization. 

Here, we assume that $V$ has bounded global generators $\psi$. This means that there exist bounded holomorphic functions $\psi_1, \ldots, \psi_k$ on $\Omega$ such that
	\begin{equation*}
	V = \{\psi_1 = \cdots = \psi_k = 0\}. 
	\end{equation*}
Then, by boundedness of generators, we can take a positive constant $M$ such that $\frac{|\psi|}{M} < 1$, i.e., $\log\frac{|\psi|}{M} \in \mathcal{F}_V$. Since $\mathcal{F}_V \neq \varnothing$, there exists the pluricomplex Green function $G_{\Omega,V}$ with poles along $V$. 

And, since $V$ is submanifold, for any $z^{\circ} \in V$, we can take the coordinate $z=(z_1,\ldots,z_n)$ near $z^{\circ}$ such that
	\begin{equation*}
	J_{\psi} := \det \frac{\partial(\psi_1, \ldots, \psi_k)}{\partial(z_{1}, \ldots, z_k)} \neq 0. 
	\end{equation*}
In particular, in the case where $V$ has bounded global generator, for any $z \in \Omega$, we can find a re-ordering of linear coordinates $z':=(z_1,\ldots,z_{k})$ and $z'':=(z_{k+1},\ldots,z_{n})$ such that $J_{\psi} = \det \frac{\partial(\psi_1, \ldots, \psi_k)}{\partial(z_{1}, \ldots, z_k)} \neq 0$. 

In this setting, we establish an analogue of the Berndtsson--Lempert type $L^2$-extension theorem. 

\begin{Thm}\label{theoremOT}
Let $\Omega$ be a bounded pseudoconvex domain in $\mathbb{C}^n$, $V$ a closed complex submanifold of $\Omega$ with codimension $k$ such that $V$ has bounded global generators $\psi=(\psi_1,\ldots,\psi_k)$ and there exists a positive constant $C$ such that $\displaystyle \frac{1}{C} \le |J_{\psi}|$ near $V$. Let $\varphi$ be a plurisubharmonic function on $\Omega$ and $G_{\Omega,V}$ the pluricomplex Green function on $\Omega$ with poles along $V$. Assume that there exists some continuous function $B$ on $\Omega$ such that  
	\begin{equation}\label{estimatebelow}
	\log |\psi(z)| - B(z) \le G_{\Omega,V}(z). 
	\end{equation}
Then for any holomorphic function $f$ on $V$ with $\displaystyle \int_V |f|^2 e^{-\varphi + 2kB} < \infty$, there exists a holomorphic function $F$ on $\Omega$ such that $F|_V = f$ and
	\begin{equation*}
	\int_\Omega |F|^2 e^{-\varphi} \le C \sigma_k \int_V |f|^2 e^{-\varphi + 2kB}, 
	\end{equation*}
where $\sigma_k$ is the volume of the unit ball in $\mathbb{C}^k$. 
\end{Thm}

\begin{Rem}
Theorem \ref{theoremOT} is obtained even if we use an another function in $\mathcal{F}_{V}$ satisfying \eqref{estimatebelow} in substitution for $G_{\Omega,V}$. 
In general, if we use an another function in $\mathcal{F}_{V}$ satisfying \eqref{estimatebelow}, the $L^2$-estimate obtained from the conclusion becomes rough.
\end{Rem}

\begin{proof}
To prove Theorem \ref{theoremOT},  we follow the argument in the proof of \cite[Theorem 2.1]{H3}. 

Let $A^2(\Omega,\varphi) = A^2(\Omega)$ be a Hilbert space of holomorphic functions $F$ on $\Omega$ with $\int_{\Omega} |F|^2 e^{-\varphi} < \infty$. We may assume that $\varphi$ is continuous. 

Let $F^{\circ} \in A^2(\Omega,\varphi)$ be an arbitrary $L^2$-extension of $f$ to $\Omega$. 
Here, we consider a sequences of domains $\Omega_j$ such that $\Omega_j$ approximates $\Omega$ from inside. 
From now on, we will discuss on $\Omega_j$. For simplicity, we omit the subscript $j$. 

Let $F_0$ be the $L^2$-minimum extension of $f$. The $L^2$-norm $||F_0||_{A^2(\Omega)}$ of $F_0$ is equal to the $L^2$-norm $\Vert F^{\circ} \Vert_{A^2(\Omega)/{A^2(\Omega) \cap \mathcal{I}_V}}$ of $F^{\circ}$, where we denote $\mathcal{I}_V := \{ F \in \mathcal{O}(\Omega) : F|_V = 0 \}$. Actually, 
	\begin{align*}
	\Vert F_0 \Vert_{A^2(\Omega)} &= \sup_{0 \neq e \in A^2(\Omega)} \frac{|(e,F_0)_{A^2(\Omega)}|}{\Vert e 		\Vert_{A^2(\Omega)}} 
	                            \\&= \sup_{0 \neq e \in \mathcal{I}_V^{\bot}} \frac{|(e,F_0)_{A^2(\Omega)}|}{\Vert e \Vert_{A^2(\Omega)}}
                                \\&= \sup_{0 \neq \xi \in A^2(\Omega)^* \cap \text{Ann} \mathcal{I}_V} \frac{|\left<\xi,F_0\right>|}	{\Vert \xi \Vert_{A^2(\Omega)^*}}
                                \\&= \sup_{0 \neq \xi \in A^2(\Omega)^* \cap \text{Ann} \mathcal{I}_V} \frac{|\left<\xi,F^{\circ}\right>|}{\Vert 	\xi \Vert_{A^2(\Omega)^*}}
                                \\&= ||F^{\circ}||_{A^2(\Omega)/{A^2(\Omega) \cap \mathcal{I}_V}},  
	\end{align*}
where we denote $\text{Ann} \mathcal{I}_V := \{ \xi \in A^2(\Omega)^* : \xi |_{A^2(\Omega) \cap \mathcal{I}_V} = 0 \}$.
Then we deal with the linear form $\xi$ in the following way. For a fixed smooth function $g$ on $V$ with compact support, we define a linear functional $\xi_g$ on $A^2(\Omega)$ by
	\begin{equation*}
	\left< \xi_g, h \right> := \sigma_k \int_V h \bar{g} e^{-\varphi+2kB}, \quad h \in A^2(\Omega, \varphi). 
	\end{equation*}
The set of such functionals $\xi_g$ is a dense subspace of $(A^2(\Omega)/{A^2(\Omega) \cap \mathcal{I}_V})^*$. Therefore, the $L^2$-norm $\Vert F^{\circ} \Vert_{A^2(\Omega)/{A^2(\Omega) \cap \mathcal{I}_V}}$ can be written as
	\begin{equation*}
	\sup_{g} \frac{|\left< \xi_g, F^{\circ} \right>|}{||\xi_g||_{A^2(\Omega)^*}}. 
	\end{equation*} 

For $p > 0$, $t \in \mathbb{C}$ with Re $t \le 0$, we define
	\begin{equation*}
	\varphi_{t,p}(z) := \varphi(z) + p \max \left\{ G(z)-\frac{\text{Re} t}{2}, 0 \right\}. 
	\end{equation*}
For any fixed $p$, by the convexity theorem in {\cite[Theorem 1.1]{Ber}}, 
	\begin{equation*}
	t \longmapsto \log||\xi_g||_{A^2(\Omega, \varphi_{t,p})^*}
	\end{equation*}
is subharmonic. In particular, $\varphi_{t,p}$ is convex in Re $t$. Therefore, we can assume that $t \in \mathbb{R}_{\le 0}$. 
The following lemma describes the asymptotic behavior of $||\xi_g||_{A^2(\Omega, \varphi_{t,p})^*}$ when $t \to -\infty$. 
We will write $||\xi_g||_{A^2(\Omega, \varphi_{t,p})^*} = ||\xi_g||_{t,p}$ for simplicity. 

\begin{Lem}[{\cite[Lemma 3.2]{BL}}]\label{asymptotic behavior}

For fixed $p>0$, it follows that 
	\begin{equation*}
	||\xi_g||_{t,p} e^{\frac{kt}{2}} = O(1). 
	\end{equation*}
when $t \to -\infty$. In particular, $||\xi_g||_{t,p} e^{\frac{kt}{2}}$ is increasing in $t$. 
\end{Lem}

Let $F_{t,p}$ be the $L^2$-minimum extension of $f$ in $A^2(\Omega, \varphi_{t,p})$. By Lemma \ref{asymptotic behavior}, 
	\begin{equation*}
	e^{-\frac{kt}{2}} ||F_{t,p}||_{A^2(\Omega, \varphi_{t,p})} = \sup_g \frac{\left< \xi_g, F^{\circ} \right>}{e^{\frac{kt}{2}}||\xi_g||_{t,p}}
	\end{equation*}
is decreasing in $t$. Therefore, it follows that
	\begin{align*}
	||F_0||_{A^2(\Omega, \varphi)} & \le e^{-\frac{kt}{2}} ||F_{t,p}||_{A^2(\Omega, \varphi_{t,p})}
	                             \\& \le e^{-\frac{kt}{2}} ||F^{\circ}||_{A^2(\Omega, \varphi_{t,p})}.
	\end{align*}
For fixed $t<0$, the right hand side of the last inequalities converges to 
	\begin{equation*}
	e^{-\frac{kt}{2}} \left( \int_{\{G < \frac{t}{2}\}} |F^{\circ}|^2 e^{-\varphi} \right)^{\frac{1}{2}}
	\end{equation*}
when $p \to \infty$. 

The subscript $j$ will be specified from here. We prepare the following lemma. 

\begin{Lem}\label{L^2estimate}
Let $\chi \ge 0$ be a continuous function on $\bar{\Omega}$ and integrable on $V$. Then, it follows that
	\begin{equation*}
	\limsup_{t \to -\infty} e^{-kt}\int_{\Omega_j \cap \Omega_t}\chi \le C \sigma_k \int_{\Omega_j \cap V} \chi e^{2kB}, 
	\end{equation*}
where we denote $\displaystyle \Omega_t = \left\{G_{\Omega,V} < \frac{t}{2} \right\}$.  
\end{Lem}

\begin{proof}
By using the assumption on $G_{\Omega,V}$, we can denote 
	\begin{equation*}
	\Omega_t = \{|\psi| < e^{\frac{t}{2} + B}\}.
	\end{equation*}
Since the submanifold $\overline{\Omega_j \cap V}$ is compact in $\mathbb{C}^n$, there exists a finite open covering $\{ U_i \}_{i=1}^N$ such that there exists a change of numbering of linear coordinates depending on $i$ so that we have 
	\begin{equation*}
	J_{\psi} = \det \frac{\partial(\psi_1, \ldots, \psi_k)}{\partial(z_{1}, \ldots, z_k)} \neq 0 \quad \text{on each $U_i$.}
	\end{equation*}	 
Let $t<0$ be a sufficiently negative number. 
Since for any point in $V$, there exists $U_i$ such that the given point in $U_i$, we have
	\begin{align}
	e^{-kt}\int_{\Omega_j \cap \Omega_t}\chi &= e^{-kt}\int_{\Omega_j \cap {\{|\psi| < e^{\frac{t}{2} + B}\}}} \chi \notag
	                                       \\&= e^{-kt}\int_{\Omega_j \cap V} dz'' \int_{\{|\psi| < e^{\frac{t}{2} + B}\}} \chi(\psi^{-1},z'') \frac{1}{|J_{\psi}|} d\psi \notag
	                                       \\&\le C e^{-kt}\int_{\Omega_j \cap V} dz'' \int_{\{|\psi| < e^{\frac{t}{2} + B}\}} \chi(\psi^{-1},z'') d\psi. \label{eqL^2}
	\end{align}
By the continuity of $\chi$, we can calculate the right-hand side of \eqref{eqL^2} as follow:
	\begin{align*}
	C e^{-kt}\int_{\Omega_j \cap V} dz'' \int_{\{|\psi| < e^{\frac{t}{2} + B}\}} \chi(\psi^{-1},z'') d\psi &\le C e^{-kt}\int_{\Omega_j \cap V} (\chi(0,z'') + \epsilon ) dz'' \int_{\{|\psi| < e^{\frac{t}{2} + B}\}} d\psi
	                                                                                                     \\&= C e^{-kt}\int_{\Omega_j \cap V} (\chi(0,z'') + \epsilon ) \sigma_k e^{2k(\frac{t}{2} + B)} dz''
	                                                                                                     \\&= C \sigma_k \int_{\Omega_j \cap V} (\chi(0,z'') + \epsilon ) e^{2kB} dz''. 
	\end{align*}
	
Take the upper limits with respect to $t$ of both sides of the above inequalities and $t \to -\infty$, we get the conclusion. 
\end{proof}

We conclude from Lemma \ref{L^2estimate} that 
	\begin{align*}
	||F_0^{(j)}||_{A^2(\Omega_j, \varphi)}^2 &\le \limsup_{t \to -\infty} e^{-kt}||F^{\circ}||_{A^2(\Omega_j \cap \{G < \frac{t}{2}\}, \varphi)}^2
	                                     \\&\le C \sigma_k \int_V |f|^2 e^{-\varphi + 2kB}.
	\end{align*}
Therefore, the $L^2$-norms $||F_0^{(j)}||_{A^2(\Omega_j, \varphi)}$ are uniformly bounded with respect to $j$. 
After taking a subsequence of $\{F_0^{(j)}\}$, it converges in $A^2(\Omega_{j_0},\varphi)$ for fixed $j_0$, thus the limit $F_0^{(\infty)} \in A^2(\Omega, \varphi)$ of it satisfies 
	\begin{equation*}
	||F_0^{(\infty)}||_{A^2(\Omega, \varphi)} \le C \sigma_k \int_V |f|^2 e^{-\varphi + 2kB}.
	\end{equation*}

\end{proof}

	
\section{Higher dimensional case of sharper estimates}

In this section, we prove the higher dimensional case of Hosono's sharper estimate by using Theorem \ref{theoremOT}.
Our setting is as follows. 

Let $\Omega$ be a bounded pseudoconvex domain in $\mathbb{C}^n$ and $V$ a closed complex submanifold of $\Omega$ with codimension $k$ such that $V$ has bounded global generators $\psi=(\psi_1,\ldots,\psi_k)$. 
Suppose that there exists a positive constant $C$ such that $\displaystyle \frac{1}{C} \le |J_{\psi}|$ near $V$. Let $\varphi$ be a plurisubharmonic function on $\Omega$ and $G_{\Omega,V}$ the pluricomplex Green function on $\Omega$ with poles along $V$. Assume that there exists some continuous function $B$ on $\Omega$ such that  
	\begin{equation*}
	\log |\psi(z)| - B(z) \le G_{\Omega,V}(z). 
	\end{equation*}
Let $\tilde{\Omega}$ be a pseudoconvex domain in $\mathbb{C}^{n+k}$ defined by 
	\begin{equation*}
	\tilde{\Omega} = \{(z,w) \in \mathbb{C}^{n+k} : z \in \Omega, |w|^2 < e^{-\frac{\varphi(z)}{k}} \} 
	\end{equation*}
and $\tilde{V}$ a closed complex submanifold of $\tilde{\Omega}$ such that
	\begin{equation*}
	\tilde{V} = \{\tilde{\psi}_1 = \cdots = \tilde{\psi}_k = 0 \}, 
	\end{equation*}
where $\tilde{\psi}_i(z,w) := \psi_i(z) $ are holomorphic functions on $\tilde{\Omega}$. 
Let $\tilde{G}$ be a pluricomplex Green function on $\tilde{\Omega}$ with poles along $\tilde{V}$ or an another function in $\mathcal{F}_V$ such that there exists continuous function $\tilde{B}(z,w)$ on $\tilde{\Omega}$ such that 
	\begin{equation*}
	\log|\tilde{\psi}(z,w)| - \tilde{B}(z,w) \le \tilde{G}(z,w).
	\end{equation*}
Then the following theorem holds. 

\begin{Thm}\label{theoremH}
Under the above setting, the following statements hold.   

$(1)$ For any holomorphic function $f$ on $V$ with $\displaystyle \int_V |f|^2 e^{-\varphi+2kB} < \infty$, there exists a holomorphic function $F$ on $\Omega$ such that $F|_V = f$ and
	\begin{equation*}
	\int_{\Omega} |F|^2 e^{-\varphi} \le C \int_{\tilde{V}} |\tilde{f}|^2 e^{2k\tilde{B}},  
	\end{equation*}
where C is a positive constant determined from Theorem \ref{theoremOT} and the holomorphic function $\tilde{f}$ on $\tilde{V}$ is defined by $\tilde{f}(z,w) := f(z)$.

$(2)$ Suppose that $\tilde{\Omega}$ is a strictly pseudoconvex domain and $-B(z)$ is a plurisubharmonic function. Then one can make the estimate in $(1)$ strictly sharper than one in Theorem \ref{theoremOT}, i.e., there exist functions $\tilde{G}$ and $\tilde{B}$ satisfying the above conditions such that the estimate
	\begin{equation*}
	\int_{\tilde{V}} |\tilde{f}|^2 e^{2k\tilde{B}} < \sigma_k \int_V |f|^2 e^{-\varphi + 2kB} 
	\end{equation*}
holds.
\end{Thm}

\begin{proof}
$(1)$ By applying Theorem \ref{theoremOT} with the trivial metric $e^{-\tilde{\varphi}} \equiv 1$ to the holomorphic function $\tilde{f}(z,w) := f(z)$ on $\tilde{V}$, we get a holomorphic function $\tilde{F}$ on $\tilde{\Omega}$  satisfying the properties that 
$\tilde{F}|_{\tilde{V}} = \tilde{f}$ and
	\begin{equation*}
	\int_{\tilde{\Omega}}|\tilde{F}|^2 \le C \sigma_k \int_{\tilde{V}} |\tilde{f}|^2 e^{2k\tilde{B}} = C \sigma_k \int_{\tilde{V}} |f|^2 e^{2k\tilde{B}}. 
	\end{equation*}
We consider a holomorphic function $F(z) := \tilde{F}(z,0)$ on $\Omega$. For any $z \in V$, we have
	\begin{equation*}
	F(z) = \tilde{F}(z,0) = \tilde{f}(z,0) = f(z), 
	\end{equation*}
i.e., $F|_V = f$ holds. For fixed $z \in \Omega$, by the mean value inequality, we have
	\begin{equation*}
	|F(z)| = |\tilde{F}(z,0)| \le \frac{1}{\sigma_k e^{-\varphi(z)}} \int_{|w|^2 < e^{-\frac{\varphi(z)}{k}}} |\tilde{F}|^2.
	\end{equation*}
Therefore, it follows that
	\begin{equation*}
	\sigma_k \int_{\Omega} |F|^2 e^{-\varphi} \le \int_{\tilde{\Omega}} |\tilde{F}|^2.
	\end{equation*}
Then the following inequality holds. 
	\begin{equation*}
	\int_{\Omega} |F|^2 e^{-\varphi} \le C \int_{\tilde{V}} |f|^2 e^{2k\tilde{B}}. 
	\end{equation*}

$(2)$ We assume that $\tilde{\Omega}$ is a strictly pseudoconvex domain. Then by \cite[Theorem D]{BT}, there exists $\tilde{u} \in PSH(\tilde{\Omega}) \cap \mathcal{C}(\bar{\tilde{\Omega}})$ such that
	\begin{equation*}
	(dd^c \tilde{u})^{n+k} = 0 \quad \text{on $\tilde{\Omega}$, and} 
	\end{equation*}
	\begin{equation*}
	\tilde{u} = -\max(\log|\tilde{\psi}|, D) \quad \text{on $\partial \tilde{\Omega}$},  
	\end{equation*}
where $D$ is a sufficiently negative constant such that
	\begin{equation*}
	\max(\log|\psi|, D) - B(z) < 0
	\end{equation*}
on $\Omega$. 
Here, we define $\tilde{G} := \log|\tilde{\psi}| + \tilde{u}$. 
On $\partial \tilde{\Omega}$, we have
	\begin{equation*}
	\tilde{G} = \log|\tilde{\psi}| - \max(\log|\tilde{\psi}|, D) \le \log|\tilde{\psi}| - \log|\tilde{\psi}| = 0.
	\end{equation*}
By the maximum principle, it follows that $\tilde{G} < 0$ on $\tilde{\Omega}$. By continuity of $\tilde{u}$, we can choose $\tilde{B} = -\tilde{u}$. 

Then it is sufficient to prove the following inequality.
	\begin{equation*}
	\int_{\tilde{V}} |f|^2 e^{2k\tilde{B}} < \sigma_k \int_V |f|^2 e^{-\varphi + 2kB}. 
	\end{equation*}
First, we calculate each side of the above inequality separately:  
	\begin{equation*}
	\int_{\tilde{V}} |f|^2 e^{2k\tilde{B}} = \int_V |f(z)|^2 \left( \int_{|w|^2 < e^{-\frac{\varphi}{k}}} e^{2k\tilde{B}(z,w)} \right), 
	\end{equation*}
	\begin{align*}
	\sigma_k \int_V |f|^2 e^{-\varphi + 2kB} &= \int_V |f(z)|^2 e^{-\varphi + 2kB(z)} \left( \int_{|w|^2 < e^{-\frac{\varphi}{k}}} e^{\varphi(z)} \right)
	                                       \\&= \int_V |f(z)|^2 \left( \int_{|w|^2 < e^{-\frac{\varphi}{k}}} e^{2kB(z)} \right).
	\end{align*}
Therefore, we need to compare the value of $\displaystyle \int_{|w|^2 < e^{-\frac{\varphi}{k}}} e^{2k\tilde{B}(z,w)}$ and $\displaystyle \int_{|w|^2 < e^{-\frac{\varphi}{k}}} e^{2kB(z)}$. 

We define a plurisubharmonic function $\tilde{u}'(z,w) := -B(z)$ on $\tilde{\Omega}$. Then the function $\tilde{u}'$ satisfies the following conditions. 
	\begin{equation*}
	(dd^c \tilde{u}')^{n+k} \ge 0 \quad \text{on $\tilde{\Omega}$}, 
	\end{equation*}
	\begin{equation*}
	\tilde{u}' = -B \quad \text{on $\partial \tilde{\Omega}$}. 
	\end{equation*} 

Since $\tilde{u}=-\max(\log|\tilde{\psi}|, D) \ge \tilde{u}'=-B$ on $\partial \tilde{\Omega}$ and $(dd^c \tilde{u})^{n+k} \le (dd^c \tilde{u}')^{n+k}$ on $\tilde{\Omega}$, \cite[Theorem A]{BT} yields $\tilde{u} \ge \tilde{u}'$ on $\tilde{\Omega}$. In particular, for any $z \in V$, it follows that
	\begin{equation*}
	\tilde{u}(z,w) \ge \tilde{u}'(z,w) = -B(z)
	\end{equation*}
on $\{|w|^2 < e^{-\frac{\varphi(z)}{k}}\}$.
Since $\tilde{u}=-D>-B$ on $\tilde{V} \cap \partial \tilde{\Omega}$ and $\tilde{u}$ is continuous near $\partial \tilde{\Omega}$, for any $z \in V$, $\{ w:|w|<e^{-\frac{\varphi(z)}{k}},\,\tilde{u}(z,w)> -B(z) \}$ has positive measure. Therefore, we have
	\begin{equation*}
	\int_{|w|^2 < e^{-\frac{\varphi}{k}}} e^{2k\tilde{B}(z,w)} < \int_{|w|^2 < e^{-\frac{\varphi}{k}}} e^{2kB(z)}.
	\end{equation*}
\end{proof}

\begin{Rem}
When $B(z)$ is a continuous function on $\Omega$ with $\log |\psi(z)| - B(z) \le G_{\Omega,V}(z)$ on $\Omega$, we can obtain a continuous function $B'(z)$ on $\Omega$ such that $-B'(z)$ is plurisubharmonic and $\log |\psi(z)| - B'(z) \le G_{\Omega,V}(z)$ on $\Omega$. In fact, by \cite[Theorem D]{BT}, there exists $-B' \in PSH(\tilde{\Omega}) \cap \mathcal{C}(\bar{\tilde{\Omega}})$ such that 
	\begin{equation*}
	(dd^c (-B'))^{n+k} = 0 \quad \text{on $\tilde{\Omega}$ and}
	\end{equation*}
	\begin{equation*}
	{-B'} = -B \quad \text{on $\partial \tilde{\Omega}$}. 
	\end{equation*}
Then, since 
	\begin{equation*}
	\log|\tilde{\psi}| - B' = \log|\tilde{\psi}| - B(z) \le G_{\Omega,V}(z) \le 0
	\end{equation*}
on $\partial \tilde{\Omega}$, from the maximum principle, it follows that $\log|\tilde{\psi}| - B' \le 0$ on $\tilde{\Omega}$. 
In particular, we have $\log|\psi| - B'(z,0) \le G_{\Omega,V}(z)$ on $\Omega$. 

But, we do not know whether we can obtain the sharper estimates after replacing the function $B$ with $B'$. 

\end{Rem}


\section{Special cases}

\subsection{Toward sharper estimates of the Ohsawa--Takegoshi $L^2$-extension theorem in terms of the Azukawa pseudometric}

The $L^2$-estimate obtained from the conclusion will be sharper if $B(z)$ is smaller, i.e., $G_{\Omega,V}(z)$ is bigger. 
Therefore, we need to take the pluricomplex Green function to obtain a better $L^2$-estimate. 
As a special case, the following result was obtained in \cite{H2}.  

\begin{Thm}[\cite{H2}]\label{Azukawa}

Let $\Omega$ be a bounded pseudoconvex domain in $\mathbb{C}^n$, $w$ a point in $\Omega$ and $\varphi$ a plurisubharmonic function on $\Omega$.  
Let $g_{\Omega, w}$ be the pluricomplex Green function on $\Omega$ with a pole at $w$ and $A_{\Omega, w}$ the Azukawa pseudometric.
We assume that there exists the limit
	\begin{equation*}
	A_{\Omega,w}(X) = \lim_{\lambda \to 0} \left( g_{\Omega,w}(w + \lambda X) - \log|\lambda| \right). 
	\end{equation*}
Then there exist a holomorphic function $f$ on $\Omega$ such that $f(w) = 1$ and
	\begin{equation*}
	\int_{\Omega} |f|^2 e^{-\varphi} \le {\rm{vol}}(I_{\Omega,w})e^{-\varphi(w)}, 
	\end{equation*}	 
where $I_{\Omega,w}$ is the Azukawa indicatrix defined by $I_{\Omega,w}:= \{ X \in \mathbb{C}^n : A_{\Omega,w}(X) < 0 \}$ and ${\rm{vol}}(I_{\Omega,w})$ is the euclidean volume of $I_{\Omega,w}$. 

\end{Thm}

\begin{Rem}
In \cite{Z}, it is shown that the assumptions of Theorem \ref{Azukawa} holds on a bounded hyperconvex domain. 
\end{Rem}

When the submanifold $V$ is $\{ z_1 = \cdots = z_k=0 \}$, we can generalize Theorem \ref{Azukawa} by using the pluricomplex Green function with poles along $V$.  

\begin{Thm}\label{Azukawa2}
Let $\Omega$ be a bounded pseudoconvex domain in $\mathbb{C}^n$, $V$ a closed complex submanifold defined by $V = \{ z_1 = \cdots = z_k = 0 \}$ and $\varphi$ a plurisubharmonic function on $\Omega$. 
Let $G_{\Omega,V}$ be the pluricomplex Green function on $\Omega$ with poles along $V$. 
We assume that there exists the limit
	\begin{equation*}
	A_{\Omega, V, w}(X) := \lim_{\lambda \to 0} \left( G_{\Omega,V}(\lambda X, w) - \log|\lambda| \right), 
	\end{equation*}
where $(0,\ldots,0,w) \in V$ and $0 \neq X \in \mathbb{C}^k$. 
We define $I_{\Omega, V, w} := \{X \in \mathbb{C}^n : A_{\Omega, V, w}(X) < 0 \}$. 
Then for any holomorphic function $f$ on $V$ with $\int_V {\rm{vol}}(I_{\Omega, V, w}) |f|^2 e^{-\varphi} < \infty$, 
there exists a holomorphic function $F$ on $\Omega$ such that $F|_V = f$ and
	\begin{equation*}
	\int_{\Omega} |F|^2 e^{-\varphi} \le \int_V {\rm{vol}}(I_{\Omega, V, w}) |f|^2 e^{-\varphi}. 
	\end{equation*}	 
\end{Thm}

\begin{proof}
It is sufficient to prove the following lemma. 

\begin{Lem}\label{L^2estimate2}
Let $\chi \ge 0$ be a continuous function on $\bar{\Omega}$. Then we have  
	\begin{equation*}
	\limsup_{t \to -\infty} e^{-kt}\int_{\Omega_t}\chi \le \int_{V} {\rm{vol}}(I_{\Omega,V,z''}) \chi ,  
	\end{equation*}
where $\displaystyle \Omega_t := \left\{G_{\Omega,V} < \frac{t}{2}\right\}$ and $(0,\ldots,0,z'') \in V$. 
\end{Lem}

\begin{proof}

For any $\delta>0$ and sufficiently negative $t$, by continuity of $\chi$, we have
	\begin{align*}
	e^{-kt}\int_{\Omega_{t-2\delta}}\chi &\le e^{-kt} \int_{V} (\chi(0, z'') + \epsilon ) dz'' \int_{\{G_{\Omega,V}(z', z'') < \frac{t}{2} -\delta \}} dz'. 
	\end{align*}
Here, for any $(0,\ldots,0,z'') \in V$, we consider the value of $\displaystyle e^{-kt} \int_{\{G_{\Omega,V}(z', z'') < \frac{t}{2} -\delta \}} dz'$. 
Using the substitution $\displaystyle z' = e^{\frac{t}{2}} \tilde{z}'$ , we have
	\begin{equation}\label{A}
	e^{-kt} \int_{\{G_{\Omega,V}(z', z'') < \frac{t}{2} -\delta \}} dz' = \int_{\{G_{\Omega,V}(e^{\frac{t}{2}}\tilde{z}', z'') - \log e^{\frac{t}{2}} < -\delta \}} d\tilde{z}' .
	\end{equation}
By the assumptions of $G_{\Omega, V}$, take the upper limits with respect to $t$ of both sides of \eqref{A}, then the right-hand side of \eqref{A} converges to something whose magnitude is at most
	\begin{equation*}
	\int_{\{A_{\Omega,V,z''}(\tilde{z}') \le -\delta \}}d\tilde{z}' .
	\end{equation*}
This value can be estimated as follow: 
	\begin{align*}
	\int_{\{A_{\Omega,V,z''}(\tilde{z}') \le -\delta \}}d\tilde{z}' &\le \int_{\{A_{\Omega,V,z''}(\tilde{z}') < 0 \}}d\tilde{z}'
	                                                              \\&= {\rm{vol}}(I_{\Omega,V,z''}). 
	\end{align*}
Therefore, by $\delta \to 0$, we can get the following inequality
	\begin{equation*}
	\limsup_{t \to -\infty} e^{-kt}\int_{\Omega_t}\chi \le \int_{V} {\rm{vol}}(I_{\Omega,V,z''}) \chi. 
	\end{equation*}
\end{proof}

By replacing Lemma \ref{L^2estimate} in the proof of Theorem \ref{theoremOT} with Lemma \ref{L^2estimate2}, we can prove Theorem \ref{Azukawa2}.
\end{proof}

\begin{Ex}
Here, for $n \ge 2$, we consider a unit ball $\Omega = \mathbb{B}^n$ in $\mathbb{C}^n$  and $\varphi(z) = -n\log(1-|z|^2)$. 
In this situation, the pluricomplex Green function $g_{\mathbb{B}^n,0}(z)$ with a pole at $0$ is equal to $\log|z|$ and the Azukawa pseudometric $A_{\mathbb{B}^n,0}(X)$ is equal to $\log|X|$. 
Since the Azukawa indicatrix is  $I_{\mathbb{B}^n,0}=\{|X| < 1\}$, therefore we have ${\rm{vol}}(I_{\mathbb{B}^n,0}) = \sigma_n$. 
On the other hand, since $\tilde{\Omega} = \{|w|^2 + |z|^2 < 1\}$ in $\mathbb{C}^{2n}$ and $\tilde{V} = \{|w| < 1\}$, the pluricomplex Green function $G_{\tilde{\Omega}, \tilde{V}}$ with poles along $\tilde{V}$ is equal to $\displaystyle \log \frac{|z|}{\sqrt{1-|w|^2}}$ and $\displaystyle A_{\tilde{\Omega},\tilde{V},w}(X)$ is equal to $\log|X| - \frac{1}{2}\log(1-|w|^2)$. 
Since $\displaystyle I_{\tilde{\Omega}, \tilde{V}, w} = \{ |X| < (1-|w|^2)^{\frac{1}{2}} \}$, we have ${\rm{vol}}(I_{\tilde{\Omega}, \tilde{V}, w}) = \sigma_n (1-|w|^2)^n$. 
Then
	\begin{align*}
	\int_{\tilde{V}} {\rm{vol}}(I_{\tilde{\Omega}, \tilde{V}, w}) &= \int_{|w|<1} \sigma_n (1-|w|^2)^n
	                                                   \\&= \sigma_n \int_{S^{2n-1}} \int_0^1 (1-r^2)^n r^{2n-1} dr dS
	                                                   \\&= \sigma_n \mu_n \int_0^1 (1-r^2)^n r^{2n-1} dr
	                                                   \\&= \frac{\sigma_n \mu_n}{2} B(n,n+1)
	                                                   \\&= \frac{\sigma_n \mu_n}{2} \frac{\Gamma(n) \Gamma(n+1)}{\Gamma(2n+1)}
	                                                   \\&= \frac{\sigma_n \mu_n}{2} \frac{(n-1)! n!}{(2n)!}, 
	\end{align*}
where $\mu_n$ is the volume of $S^{2n-1}$, $B$ is the Beta function and $\Gamma$ is the Gamma function. 
For any $n \ge 2$, since $\displaystyle \mu_n = \frac{2\pi^n}{\Gamma(n)} = \frac{2\pi^n}{(n-1)!}$, it follows that 
	\begin{equation*}
	\displaystyle \frac{\mu_n}{2} \frac{(n-1)! n!}{(2n)!} < 1. 
	\end{equation*}
Therefore, in this situation, it follows that $\int_{\tilde{V}} {\rm{vol}}(I_{\tilde{\Omega}, \tilde{V}, w}) < {\rm{vol}}(I_{\mathbb{B}^n,0})$. 
From this observation, we can expect that the sharper estimates of Ohsawa--Takegoshi $L^2$-extension theorem in terms of the Azukawa pseudometric holds. 
\end{Ex}

\subsection{Radial case in $\mathbb{C}^n$}

In \cite{H1}, Hosono obtained the $L^2$-minimum extension of the function $1$ on the subvariety $\{0\}$ by applying Hosono's sharper estimate to the case where $\Omega = \{|z|<1\}$ is a unit disc in $\mathbb{C}$ and $\varphi$ is a radial subharmonic function, i.e., $\varphi(z) = \varphi(|z|)$.
Similarly, in this subsection, we obtain the $L^2$-minimum extension of holomorphic functions $f$ on $V$ in the setting where $\Omega = \mathbb{B}^n$ is a unit ball in $\mathbb{C}^n$, $V$ is a closed submanifold defined by $V = \{z_1 = \cdots = z_k = 0\} = \{z'=0\}$ and $\varphi$ is a radial plurisubharmonic function with respect to $V$, i.e., $\varphi(z) = \varphi(|z'|)$, by applying Theorem \ref{theoremH} in this setting. 

For $z \in \Omega$, we denote $z=(z',z'')$ where $z'=(z_1,\ldots,z_k), z''=(z_{k+1},\ldots,z_n)$. 
We may assume that $\varphi(0) = 0$. 
Under the above setting, we can write $\varphi(z) = k u(\log|z'|^2)$ where $u$ is a convex increasing function on $\mathbb{R}_{<0}$. 
Assume that $u$ is strictly increasing and for fixed $z''$, $\displaystyle \lim_{t \to \log(1-|z''|^2)-0} u(t) = \infty$. 
Define the plurisubharmonic function $\psi$ by
	\begin{equation*}
	\psi(w) := -\frac{1}{2}u^{-1}(-\log|w|^2).  
	\end{equation*}
Then

\begin{Prop}\label{eq}
For fixed $z''$, we have
\begin{equation}\label{4-1}
\int_{|z'|^2+|z''|^2<1} e^{-\varphi} = \int_{|w|<1} e^{-2k\psi(w)}. 
\end{equation}
\end{Prop}

\begin{proof}
Let $\mu_{k}$ be the volume of $S^{2k-1}$. 
First, we calculate the left hand side of \eqref{4-1}. 
Using the substitution $2\log r = t$, we have
	\begin{align}
	\int_{|z'|^2+|z''|^2<1} e^{-\varphi} &= \int_{|z'|^2<1-|z''|^2} e^{-k u(\log|z|^2)} \notag
    	                               \\&= \mu_k \int_0^1 e^{-k u(\log r^2)} r^{2k-1} dr \notag
    	                               \\&= \frac{\mu_k}{2} \int_{-\infty}^{\log(1-|z''|^2)} e^{-k u(t)}e^{kt}dt. \label{4-2}
	\end{align}
Next, we calculate the right-hand side of \eqref{4-1}. 
At first, using the substitution $2k\log r = t$, we have
	\begin{align}
	\int_{|w|<1} e^{-2k\psi} &= \int_{|w|<1} e^{k u^{-1}(-\log|w|^2)} \notag
    	                   \\&= \mu_k \int_0^1 e^{k u^{-1}(-\log r^2)} r^{2k-1} dr \notag
	                       \\&= \frac{\mu_k}{2k} \int_{-\infty}^0 e^{k u^{-1}(-\frac{t}{k})}e^{t} dt. \label{4-3}
	\end{align}
Then letting $u^{-1}(-\frac{t}{k}) = s$, we see that the right hand side of \eqref{4-3} is equal to
	\begin{equation}\label{4-4}
	\frac{\mu_k}{2} \int_{-\infty}^{\log(1-|z''|^2)} e^{ks}e^{-k u(s)} u'(s) ds. 
	\end{equation}
And finally, using the substitution $ks-ku(s)=q$, we calculate the difference between \eqref{4-2} and \eqref{4-4} as follows:
	\begin{align*}
	&\frac{\mu_k}{2} \int_{-\infty}^{\log(1-|z''|^2)} e^{-ku(t)}e^{kt}dt - \frac{\mu_k}{2} \int_{-\infty}^{\log(1-|z''|^2)} e^{ks}e^{-ku(s)} u'(s) ds 
	\\&= \frac{\mu_k}{2k} \int_{-\infty}^{\log(1-|z''|^2)} k(1-u'(s))e^{ks-ku(s)}ds
    \\&= \frac{\mu_k}{2k} \int_{-\infty}^{-\infty} e^q dq = 0. 
	\end{align*}
\end{proof}

We define $\tilde{G}(z,w) := \log|z| + \psi(w)$. 
From Proposition \ref{eq} and Theorem \ref{theoremOT}, we infer that for any holomorphic function $f$ on $V$, there exists a holomorphic function $F$ on $\mathbb{B}^n$ such that $F|_V = f$ and
	\begin{align*}
	\int_{\mathbb{B}^n} |F|^2 e^{-\varphi} &\le \int_{|z'|^2+|z''|^2<1} |f(0,z'')|^2 e^{-\varphi}.
	\end{align*}
Therefore, in this case, by the above inequality and the mean value inequality, we can obtain that the $L^2$-minimum extension with respect to $\varphi$ is $F(z) = f(0,z'')$ in holomorphic functions on $\mathbb{B}^n$ with $F|_V=f$. 

In the general case, for any $\epsilon > 0$, we define $ku_{\epsilon}(\log|z'|^2) := \varphi(z) - \epsilon \log (1-|z''|^2-|z'|^2)$.
Then $u_{\epsilon}$ is a strictly increasing and satisfies that for fixed $z''$, $\displaystyle \lim_{t \to \log(1-|z''|^2)-0} u_{\epsilon}(t) = \infty$. 
Therefore, for any $\epsilon>0$, we can obtain that the $L^2$-minimum extension with respect to $ku_{\epsilon}(\log|z|^2)$ is $F(z) = f(0,z'')$ in holomorphic functions on $\mathbb{B}^n$ with $F|_V=f$. 
Finally, by $\epsilon \to 0$, we get the conclusion. 	

	\vskip3mm
	{\bf Acknowledgment. } The author would like to thank Prof.\ Ryoichi Kobayashi and Dr.\ Genki Hosono for valuable comments. 
	
\bibliographystyle{plain}

\end{document}